\documentclass[10pt]{amsart}
\usepackage{a4,tikz,enumerate,dsfont,fancyhdr,colonequals,mathrsfs}
\usepackage[pdftex]{hyperref}
\numberwithin{equation}{section}
\newtheorem{theo}{Theorem}

\newtheorem{lemma}[theo]{Lemma}

\newtheorem{cor}[theo]{Corollary}

\newtheorem{rem}[theo]{Remark}

\newtheorem{exa}[theo]{Example}

\newtheorem{imbed}[theo]{Lemma}

\newtheorem{esti}[theo]{Remark}
\numberwithin{theo}{section}

 \def\mV{\mathsf{V}}
 \def\mE{\mathsf{E}}

 \def\mv{\mathsf{v}}
 \def\me{\mathsf{e}}
 \def\mw{\mathsf{w}}
 \def\mf{\mathsf{f}}

\title{Gradient systems on networks}
\author{Delio Mugnolo}
\author{Ren\'e Pr\"opper}
\thanks{}

\begin{document}

\begin{abstract}
We consider a class of linear differential operators acting on vector-valued function spaces with general coupled boundary conditions. Unlike in the more usual case of so-called \emph{quantum graphs}, the boundary conditions can be nonlinear. After introducing a suitable Lyapunov function we prove well-posedness and invariance results for the corresponding nonlinear diffusion problem.
\end{abstract}

\maketitle

\section{Introduction}

Following earlier intuitions, K.\ Ruedenberg and C.\ Scherr developed in 1953 a new technique with the aim of studying electronic properties of conjugated bond systems, and in particular of aromatic molecules in~\cite{RueSch53}. Their idea was to set up a Schr\"odinger equation acting on a quasi-1-dimensional domain that can be schematised as a network of atoms, the network's edges being the chemical bonds. Their results aroused broad interest in the community of quantum chemists and marked the birth of the so-called \emph{free-electron network model}.

Over two decades later, when most solid state physicists had already turned back to the original, computationally more feasible discrete tight binding approximations, the free-electron model began to be studied by analysts and theoretical physicists. Among the first results we mention well-posedness results for the heat equations on networks obtained by G.\ Lumer in~\cite{Lum80}.
Further interesting results followed soon: among others, F.\ Ali Mehmeti, J.\ von Below, P.\ Exner, S.\ Nicaise, Yu.\ V.\ Pokorny\u{\i} and J.-P.\ Roth extended Lumer's results considering more and more general node conditions, providing interesting descriptions of the spectrum, discussing nonlinear and/or higher dimensional problems and establishing an interplay with quantum physics and theoretical mechanics. We refer to~\cite{PokBor04,Kuc08} for a survey of these early investigations.

This topic has finally gone mainstream in the late 1990s, when T. Kottos and U. Smilansky have observed in~\cite{KotSmi97,KotSmi99} that models based on differential operators on metric graphs can play a fundamental role in the theory of quantum chaos. Ever since, network-based differential models for Schr\"odinger equations (and, by extension, also diffusion, Dirac and Pauli ones) have been commonly referred to as ``quantum graphs'' in the literature.

Another interesting development has begun with~\cite{KosSch99,Har00,Kuc04,Pan05}, where a variety of nonstandard boundary conditions for partial differential equations on graphs have been described. In this general sense, a quantum graph is just a particular way to look at a vector-valued diffusion equation with coupled boundary conditions, and the associated elliptic problems can be studied by means of classical Sturm--Liouville theory: observe that the above formalisms contained in the above papers are essentially just simple cases of the general framework introduced in~\cite{SchSch65,SchSch66}. Though, such representations of coupled boundary conditions have paved the road for the development of abstract functional analytical methods for the treatment of general diffusion problems. It has been observed by several authors that quantum graphs represent a handy source of examples for unusual or even pathological behaviours of a diffusion equation. 

\bigskip
Nonlinear Robin-type boundary conditions for  parabolic equations on domains are relatively common in the literature. They model an outgoing flow that depends nonlinearly on the temperature or the density at the domain's boundary. However, they have been seldom considered in the framework of quantum graphs. To the best of our knowledge, in the specific case of networks they have been treated only in~\cite{Bel91,Bel93}. The more general case of differential inclusions on ramified spaces is treated in~\cite{AliNic93}, while nonautonomous semilinear parabolic systems have been treated by several autors following the seminal article~\cite{Ama88}. 

The aim of this note is to introduce a general theory of 1-dimensional parabolic systems featuring nonlinear boundary conditions. Thereto we apply the theory of gradient systems and a formulation of (coupled) boundary conditions that is strongly inspired by the aforementioned works by P.\ Kuchment. In Section~2 we are going to prove well-posedness for a large class of quantum graph-like systems with nonlinear boundary conditions in variational form. In Section~3 we are going to discuss some of their qualitative properties and $L^p$-well-posedness. We conclude the paper by briefly discussing the issue of diffusion equations with nonlinear dynamic boundary conditions in Section~4.

\bigskip
Finally, we emphasise that while the theory of quantum graphs has been the source of inspiration for the present investigation, it is misleading to connect our results to those in that field. In fact, while in the linear case well-posedness and qualitative properties for diffusion and Schr\"odinger equations on networks (as well as on domains) often come in pairs, in the nonlinear case most such connections fail to hold. The reason is that, of course, spectral methods lose their strength and have to be replaced by a hard analysis approach which is specific to the considered class of differential equations. In fact, our approach is based on the theory of gradient systems, which is currently pretty much bound to real Banach spaces. This seems to prevent any investigation of Schr\"odinger equations, even of damped ones. This is why we prefer to refer to previous investigations on diffusion equations with nonlinear coupled boundary conditions, although Schr\"odinger equations with concentrated nonlinearities that may look formally similar to ours have also been treated in the literature, see e.g.~\cite{AdaTet01}.

\section{General setting and well-posedness results}\label{sec2}

Let $H$ be a separable, real Hilbert space and $Y$ be a closed subspace of $H\times H$. Let $T>0$. Throughout this section we discuss the vector-valued diffusion equation
\begin{equation*}\tag{AV}
\left\{\begin{array}{rcll}
\frac{\partial}{\partial t}{u}(t,x)&=& \frac{\partial^2}{\partial x^2} u(t,x)+\psi(t,x) ,\qquad &t\in [0,T],\; x\in (0,1),\\
u(t)_{|\{0,1\}}&\in& Y, &t\in [0,T],\\
\frac{\partial u(t)}{\partial\nu}+\phi\left((u(t)_{|\{0,1\}}\right)&\in & Y^\perp, &t\in [0,T],\\
u(0)&=&u_0,\\
\end{array}
\right.
\end{equation*}
for an unknown $u:[0,T] \times (0,1)\to H$ with inhomogenous term  $\psi:[0,T] \times (0,1) \to H$ and a possibly nonlinear function $\phi:Y \to Y$, occurring in the  boundary condition. Here we denote
$$f_{|\{0,1\}}:=\begin{pmatrix}
f(0)\\ f(1)
\end{pmatrix}
\qquad\hbox{and}\qquad
\frac{\partial f(t)}{\partial\nu}:=\begin{pmatrix}
-f'(0)\\ f'(1)
\end{pmatrix}$$
for any $f:[0,1]\to H$ smooth enough. If $\phi$ is a linear mapping and $H$ is finite dimensional, it has been shown in~\cite{Kuc04} that the boundary conditions appearing in $\rm(AV)$ are the most general ones leading to self-adjoint diffusion operators.

\begin{exa}\label{zhazhe}
One of easiest nontrivial parabolic systems that can be written in the form $\rm(AV)$ is possibly that considered in~\cite{ZhaZhe03}. Restricting for the sake of simplicity to the 1-dimensional case, their boundary conditions for the unknowns $u_1,u_2$ read
$$
u_1'(i)=(-1)^{i+1} \lambda_1 e^{p_1 u_1(i)+q_1 u_2(i)},\qquad u_2'(i)=(-1)^{i+1} \lambda_2 e^{p_2 u_1(i)+q_2 u_2(i)},\qquad i=0,1,$$
for some positive parameters $\lambda_i,p_i,q_i$, $i=1,2$. Clearly, it can be rewritten as
$$u(t)_{|\{0,1\}}\in Y,\qquad \frac{\partial u(t)}{\partial\nu}+\phi\left((u(t)_{|\{0,1\}}\right)\in  Y^\perp$$
letting $H:=\mathbb R^2$, $Y:=\{0\}$ and $\phi:=(\phi_1,\phi_2)$, where
$$\phi_1\begin{pmatrix}
z_1\\ z_2
\end{pmatrix}:=
\phi_2\begin{pmatrix}
z_1\\ z_2
\end{pmatrix}:=
\begin{pmatrix}
\lambda_1 e^{p_1 z_1+q_1 z_2}\\ \lambda_2 e^{p_2 z_1+q_2 z_2}
\end{pmatrix}.$$
A thorough blow-up analysis of this system has been performed in~\cite{ZhaZhe03} and in many subsequent papers.
\end{exa}     
     
\begin{exa}\label{quantumgr}
Let $Y={\rm Range}\,\tilde{ I}$, where $I$ is the $n\times m$ (signed) incidence matrix of a finite, simple, directed graph with node set ${\mV}=\{\mv_1,\ldots,\mv_n\}$ and edge set $\mE=\{\me_1,\ldots,\me_m\}$, ${ I}^+,{ I}^-$ are the matrices whose entries are the positive and negative parts of the entries of $ I$ and
\begin{equation}\label{deftilde}
\tilde{ I}:=\begin{pmatrix} ({ I}^+)^T \\({ I}^-)^T\end{pmatrix}.
\end{equation}
If $\phi\equiv 0$, then the boundary conditions of $\rm(AV)$ agree with conditions of continuity/Kirchhoff-type, cf.~\cite{KraMugSik07}. More generally, if $\phi=\tilde{I}D^{-1}B(\tilde{I}D^{-1})^T$, they agree with so-called \emph{$\delta$-coupling conditions}
$$
\begin{array}{rcll}
u_\me(t,\mv)&=&u_\mf (t,\mv)=:u_\mv(t), &t\in [0,T],\; \me\sim \mv\sim\mf,\; \mv\in\mV,\\
\sum_{\me\sim\mv} \frac{\partial u(t)}{\partial\nu}(t,\mv)&=&\sum_{\mw\in \mV} b_{ \mv\mw} u_\mw(t),\qquad &t\in [0,T],\;\mv\in\mV,\\
\end{array}
$$
associated with an $n\times n$-matrix $B=(b_{\mv\mw})$, cf.~\cite{Mug07}. Here we write $\me\sim \mv$ if the edge $\me$ is incident in the node $\mv$, $D$ is the diagonal matrix of vertice degrees and $B$ is a $n\times n$-matrix.
\end{exa}

\begin{exa}
In the case of $H={\mathbb R}^m$ and 
$$Y:=\langle 1\rangle\times H=\{(c,c,\ldots,c,a_1,a_2,\ldots,a_m)\in H\times H:a_1,a_2,\ldots,a_m,c\in \mathbb R\},$$ 
$\rm(AV)$ reduces to a diffusion problem on a metric star with $m$ edges with nonlinear $\delta$-coupling conditions in the origin and Neumann boundary conditions in the external nodes. 

For $m=2$, the arising boundary conditions agree with those considered in~\cite{AdaTet01}. There, $H^1$-well-posedness for a Schr\"odinger (instead of heat) equation and concentrated nonlinearity (attractive or weakly repulsive interactions) in the origin has been proved.

A problem similar to $\rm{(AV)}$ has been discussed also in~\cite{CocGar10}, as the authors search for the viscosity solution of a hyperbolic system that arises in the Lighthill--Whitham traffic model. They show that a semilinear heat equation with concave nonlinearity on a star has a unique solution, for all small initial data in $H^1$.
\end{exa}

We incorporate the first boundary condition of (AV) in the Banach spaces
$$H^1_Y:=\left\{f\in H^1(0,1;H): f_{|\{0,1\}}\in Y\right\}\quad \hbox{resp.}\quad C_Y:=\left\{f\in C([0,1];H): f_{|\{0,1\}}\in Y\right\}.$$
For the inner product and the norm in $H_Y^1$, inherited from $H^1(0,1;H)$, we write $\langle \cdot\,,\, \cdot \rangle_{H^1_Y}$ resp. $\|\cdot\|_{H^1_Y}$. For the norm in $C_Y$, inherited from $C([0,1];H)$, we write $\|\cdot\|_{C_Y}$.
 We also write
\begin{equation*} X_2:=L^2(0,1;H)\;\text{ and }\; \|\cdot \|_2 :=\|\cdot \|_{L^2(0,1;H)}. \end{equation*}

Our aim is to re-write (AV) as a nonlinear abstract Cauchy problem -- more precisely, as a gradient system. To begin with, we assume that $\phi=\nabla \Phi$ for a continuously differentiable function $ \Phi: Y \to \mathbb R$ and introduce the functionals
$${\mathcal E}_0(f):=\frac{1}{2}\int_0^1 \|f^\prime(x)\|_H^2 dx, \qquad f\in H^1_Y,$$
$${\mathcal E}_1(f):=\Phi(f_{|\{0,1\}}), \qquad f\in C_Y,$$ 
and
\begin{equation} \label{lyap} {\mathcal E}:={\mathcal E}_0+{\mathcal E}_1:H^1_Y\to {\mathbb R}. \end{equation}

\begin{imbed}  \label{imbed}
The Hilbert space $H^1_Y$ is densely and continuously embedded into $X_2$.
\end{imbed} 

\begin{proof}
The assertion follows from the inclusions $H_0^1(0,1;H) \subset H^1_Y \subset H^1(0,1;H)$, since $H_0^1(0,1;H)$ is densely and $H^1(0,1;H)$ is continuously embedded into $X_2$.
\end{proof}

\begin{esti}
We observe that
\begin{equation} \label{esti} 
\|f_{|\{0,1\}} \|_{H \times H} \leq \sqrt{2} \|f \|_{\infty} \leq C \|f \|_{H^1_Y}, \quad f\in H_Y^1, 
\end{equation}
for some $C>0$. The former inequality follows from  
$$\|f_{|\{0,1\}} \|^2_{H \times H} =\|f(0)\|^2_{H} + \|f(1)\|^2_H \leq 2 \|f \|^2_{\infty},$$
and the latter from the continuous embedding of $ H^1(0,1;H)$ into $C(0,1;H)$.
\end{esti}

Inequality \eqref{esti} implies that $\mathcal E: H_Y^1 \to \mathbb R$ defined in~\eqref{lyap} is a continuously differentiable function. We recall that the gradient $\nabla \mathcal{E}$ with respect to the Hilbert spaces $H^1_Y$ and $X_2$ is the (generally nonlinear) operator $\nabla{\mathcal E}:D(\nabla{\mathcal E})\to X_2$ given by 
\begin{align*}
D(\nabla \mathcal{E})&:=\{f \in H_Y^1: \exists \, \xi \in X_2 \text{ s.t. } \mathcal{E}^\prime(f)\theta =\langle \xi\, ,\, \theta \rangle_{X_2} \;\forall \theta \in H_Y^1\}, \\
 \nabla \mathcal{E}(f)&:=\xi,
\end{align*} 
where ${\mathcal E}':H^1_Y \to (H^1_Y)'$ denotes the derivative of $\mathcal E$.

Next we are going to show that the initial-boundary value problem (AV) corresponds to the gradient system below with respect to $H^1_Y$ and $X_2$ 
\begin{equation}\tag{GS}
\left\{
\begin{array}{rcl}
\dot{u}(t)+\nabla {\mathcal E}(u(t))&=&\Psi(t),\qquad t\in [0,T],  \\
 u(0)&=&u_0,   
\end{array}
\right.
\end{equation}
where $\Psi(t):=\psi(t,\cdot)$.
 
 To this end, we observe first that the derivative of ${\mathcal E}$ is given by 
$$ {\mathcal E}^\prime(f) \theta = \int_0^1 \langle f^\prime(x)\,,\,\theta^\prime(x)\rangle_H \, dx + \langle \phi(f_{|\{0,1\}})\,,\,\theta_{|\{0,1\}} \rangle_{H \times H}, \quad f,\theta \in H^1_Y.$$
Take $f \in D(\nabla {\mathcal E}) \subset H^1_Y$ and choose $\theta \in  H^1_0(0,1;H)$. Then, by definition of $D(\nabla {\mathcal E})$, there exists $\xi\in X_2$ such that
$${\mathcal E}^\prime(f) \theta = \int_0^1 \langle f^\prime(x)\,,\,\theta^\prime(x)\rangle_H \, dx  \stackrel{!}{=}\int_0^1 \langle \xi(x) \,,\, \theta(x) \rangle_H dx.$$
We infer $\xi=-f^{\prime \prime}$ and $f \in H^2(0,1;H)$. For arbitrary $\theta \in H^1_Y$ we obtain, applying integration by parts,
\begin{align*} {\mathcal E}^\prime(f) \theta &=- \int_0^1 \langle f^{\prime \prime}(x)\,,\,\theta(x)\rangle_H \,dx  + \langle \frac{\partial f}{\partial\nu} \,,\,\theta_{|\{0,1\}} \rangle_{H\times H} +\langle \phi(f_{|\{0,1\}})\,,\,\theta_{|\{0,1\}}\rangle_{H\times H}\\ & \stackrel{!}{=}\int_0^1 \langle \xi(x)\,,\,\theta(x)\rangle_H \,dx =-\int_0^1 \langle f^{\prime \prime}(x)\,,\,\theta(x)\rangle_H \,dx.\end{align*}
Since one can find for every $y \in Y$ a $\theta \in H^1_Y$  with $\theta_{|\{0,1\}} = y$, it follows that $\frac{\partial f}{\partial\nu}+\phi(f_{|\{0,1\}}) \in  Y^\perp$.\\
If, on the other hand, $u(t,\cdot) \in H^1_Y \cap H^2(0,1;H)$, $t \in [0,T]$ satisfies the boundary condition $\frac{\partial u(t)}{\partial\nu}+\phi\left(u(t)_{|\{0,1\}}\right) \in  Y^\perp$, the above calculation yields $ u(t) \in D(\nabla {\mathcal E})$.\\

 The next Lemma is of rather general nature and similar to~\cite[Lemma~2.1]{Mug08}. We recall some definitions and well-known facts: Let $V,X$ be  Hilbert spaces such that $V\hookrightarrow X$, i.e. $V$ is densely and continuously embedded into $X$. A function $\mathcal E: V \to \mathbb R$ is called \emph{coercive} if for every $c \in \mathbb R$ the sublevel set $\{ f \in V: {\mathcal E}(f) \leq c \}$ is bounded in $V$. It is called \emph{$X$-elliptic} if  $\mathcal{E}^\omega(\cdot):={\mathcal E}(\cdot) + \omega\|\cdot\|_X^2 $ is convex and coercive for some $\omega \ge 0$. Obviously $\mathcal E$ is coercive if ${\mathcal E}(\cdot)\ge \alpha\|\cdot\|_V^2- \beta$ for some $\alpha > 0$ and $\beta  \ge 0$. For quadratic forms the converse also holds true with $\beta=0$.   

\begin{lemma}\label{perturbd}
Let $V,X$ be as above. Let $ {\mathcal E}_0:V\to\mathbb R$ be a quadratic form and $ {\mathcal E}_1:V \to \mathbb R$ be convex and satisfy
\begin{equation}\label{growthcond}
{\mathcal E}_1(f)\ge -k \|f\|^2_{X_\epsilon} - \beta \qquad \hbox{for all }f\in V \hbox{ and some } \beta, k \geq 0,
\end{equation}
where $ X_\epsilon$ is some Banach space such that $ V\hookrightarrow  X_\epsilon\hookrightarrow  X$ and verifying the interpolation inequality
\begin{equation*}
\|f\|_{X_\epsilon}\leq M_\epsilon \|f\|_{V}^\epsilon \|f\|_{X}^{1-\epsilon},\qquad f\in V,
\end{equation*}
for some $\epsilon\in [0,1)$ and some $M_\epsilon>0$. Then ${\mathcal E}_0+{\mathcal E}_1$ is $X$-elliptic if $ {\mathcal E}_0$ is $X$-elliptic.

Condition~\eqref{growthcond} is in particular satisfied if additionally $\mathcal{E}_1$ is continuous.
\end{lemma}

\begin{proof}
First  we estimate $\mathcal{E}_1$ by
\begin{align*} {\mathcal E}_1(f) &\ge -k \|f\|^2_{X_\epsilon} - \beta \\ &\geq -kK\|f\|_V\|f\|_{X_\epsilon}-\beta \\ &\geq -kKM_\epsilon \|f\|_V^{1+\epsilon}\|f\|_X^{1-\epsilon} - \beta\\
 &\geq -kKM_\epsilon(\delta \|f\|_V^2 + C(\delta)\|f\|_X^2) - \beta \end{align*}
where $K$ is a constant resulting from the continuous embedding of $V$ into  $X_\epsilon$. The last estimation follows applying Young's inequality $$xy \leq \delta x^p + C(\delta)y^{p'},\qquad x,y\ge 0,$$ 
 with $p=2/(1+ \epsilon)$ and $1/p + 1/p'=1$.\\  
Assume now $\mathcal{E}_0$ to be $X$-elliptic. Let $f\in V$. Then for appropriate $\alpha > 0$ and $\omega \geq 0$ we have
\begin{equation*} \label{star} {\mathcal E}_0(f) + \omega \|f \|_X^2\geq \alpha \|f\|_V^2 .    \end{equation*}
Choosing $\delta$ small enough, we obtain
\begin{align*}(\mathcal{E}_0 +\mathcal{E}_1)(f) &\geq \alpha' \|f\|_V^2 - \omega' \|f \|_X^2  -\beta  \end{align*}
where $0 < \alpha' := \alpha - kKM_\epsilon \delta$ and $0\le \omega' := \omega + kKM_\epsilon C(\delta)$. Furthermore, $(\mathcal{E}_0 +\mathcal{E}_1)^{\omega'}$ is convex if $\mathcal{E}_0^\omega$ is convex.

The last assertion results from the Hahn-Banach theorem or directly as following. Let $\mathcal{E}_1(0)=- b$ and $\epsilon >0$. Then there exists a $\delta > 0$, such that $\mathcal{E}_1(f) \ge -b -\epsilon$ for all $f$ with $\|f \| \le \delta$. Now convexity of $\mathcal{E}_1$ yields for every $f$ with $\| f\| \geq  \delta$ 
$$\frac{\delta}{\|f\|}\mathcal{E}_1(f)+ (1-\frac{\delta}{\|f\|})\mathcal{E}_1(0) \geq \mathcal{E}_1(\frac{\delta}{\|f\|}f) \geq -b -\epsilon $$ and together with $\mathcal{E}_1(0)=-b$ we get $\mathcal{E}_1(f) \geq -\frac{\epsilon}{\delta}\|f\|-b$ whenever $\|f\| \geq \delta$. Hence, we can choose $\gamma=b+\epsilon$ and $\kappa=\frac{\epsilon}{\delta}$ to obtain
$$\mathcal{E}_1(f)\ge -k \|f\|_{X_\epsilon} - \gamma \qquad \hbox{for all }f\in V \hbox{ and some } \gamma, \kappa \geq 0,$$
which clearly implies~\eqref{growthcond}.
\end{proof}

Now we can state the main result of this section. 

\begin{theo}\label{wellp}
Let $\phi:Y\to Y$ be a function that satisfies $\phi=\nabla \Phi$ for some convex $\Phi\in C^1(Y,{\mathbb R})$ and that maps bounded sets into bounded sets (what it does automatically if $Y$ is finite dimensional). 
 Then for all initial data $u_0 \in H^1_Y$ and all $\Psi \in L^2(0,T;X_2)$ there exists a unique solution $u\in H^1(0,T;X_2)\cap L^\infty(0,T;H^1_Y)$ to $\rm{(AV)}$. If $\Psi\equiv 0$, then the solution also belongs to $L^2(0,T;H^2(0,1;H))$ for every $T \geq 0$ and even to $L^2(0,\infty;H^2(0,1;H))$ if in addition $\Phi$ and therewith $\mathcal{E}$ is bounded from below. 
\end{theo}

These strong regularity properties of the solution are the main advantage of the approach based on gradient systems over the more general one based on maximal monotone operators as in~\cite{AliNic93,FavGolGol06}.
For the proof of Theorem \ref{wellp} we need the following general result based on~\cite[Thm.~2.1.2 bis, p.~163]{Lio69}, see also~\cite[Thm. 8.1]{ChiFas10}\footnote{ Observe that assuming the compactness of the embedding $V\hookrightarrow H$, as done in~\cite{ChiFas10}, is not necessary in our case, since the metric on $H^1_Y$ is constant.}.

\begin{theo}\label{gengrad}
Suppose that $V$ is a reflexive, separable real Banach space, which is densely and continuously embedded into the separable, real Hilbert space $X$. Suppose that $\mathcal E:V\to\mathbb R$ is a continously differentiable function such that the derivative $\mathcal E^\prime:V \to V^\prime$ maps bounded sets into bounded sets. Assume that  $\mathcal E$ is $X$-elliptic.
Then, for all $T  \in (0,\infty)$, all $\Psi \in L^2(0,T;X)$ and all initial data $u_0 \in V$ there exists a unique solution $u \in H^1(0,T;X) \cap L^\infty(0,T;V)$ of the gradient system~$\rm(GS)$. For this solution, the energy inequality
\begin{equation}\label{eneine}  
\int_0^t \| \dot{u}\|_H^2 + \mathcal{E}(u(t)) \leq \mathcal{E}(u_0) + \int_0^t \langle \Psi\,,\, \dot{u} \rangle_H ,\qquad t\in [0,T],
\end{equation}
holds.
\end{theo}

\medskip
\noindent
\emph{Proof of Theorem \ref{wellp}.} 
The last assertion of Theorem~\ref{wellp} follows immediately from the energy inequality~\eqref{eneine}, because $\dot{u}{(t)}=u''(t)$ for all $t\ge 0$. It remains to prove that $\mathcal{E}^\prime$ maps bounded sets into bounded sets and that $\mathcal E$ is $X_2$-elliptic.

To begin with, let  $f \in H^1_Y$ with $\|f \|_{H^1_Y} \leq K$. By~\eqref{esti} we know that $\| f_{|\{0,1\}}\|_{H \times H} \leq CK$, where $C$ is the same constant  as in \eqref{esti}. Let $m$ be an upper bound for $\|\phi(\cdot)\|_{H \times H}$ on the bounded set $\{y \in Y: \|y\|_{H \times H}\leq CK\}$. It follows that
\begin{equation*} \|\mathcal{E}^\prime(f)\| \begin{aligned}[t] &= \sup_{\|\theta\|_{H^1_Y}=1}|\int_0^1 \langle f^\prime(x)\,,\,\theta^\prime(x)\rangle_H \, dx + \langle \phi(f_{|\{0,1\}})\,,\,\theta_{|\{0,1\}} \rangle_{H \times H}| \\
&\leq \|f\|_{H^1_Y} \| \theta \|_{H^1_Y} +\|\phi(f_{|\{0,1\}})\|_{H \times H}\|\theta_{|\{0,1\}}\|_{H \times H}  \\
& \leq K + mC.   \end{aligned} \end{equation*}


Finally, in order to show $X_2$-ellipticity we use Lemma~\ref{perturbd}. Since $\Phi$ is convex and continuous we have
$$ \mathcal{E}_1(f)=\Phi(f_{|\{0,1\}})\geq -k \|f_{|\{0,1\}}\|^2_{H \times H} -\beta \geq -kC^2\|f\|^2_{C_Y} -\beta, \quad f \in H_Y^1,$$ for some $k,\beta \in \mathbb{R}$. Moreover, 
$$\|f\|^2_{C_Y} \leq 4\|f\|_{H^1_Y}\|f\|_2,\quad f \in H_Y^1, \qquad \text{and} \qquad H_Y^1 \hookrightarrow C(0;1;H) \hookrightarrow X_2.$$ Thus, Lemma \ref{perturbd} yields the claim.
\qed
\begin{rem}
The system in Example~\ref{zhazhe} becomes a gradient system 
for the choice $p_1=p_2$, $q_1=q_2$ and $\lambda_2=\lambda_1 q_1/p_1$; furthermore, if $\lambda_1 \leq 0$, it fulfils the requirements of Theorem~\ref{wellp} and possesses a global solution. For $\lambda >0$, on the contrary, the system undergoes a blow-up in finite time as was shown in~\cite{ZhaZhe03}. 
\end{rem}
\begin{rem}
Observe that if $u\in H^1(0,T;X_2)\cap L^\infty(0,T;H^1_Y)$, then its $C([0,T];X_2)$-representative satisfies
\begin{equation}
\label{boundeq}
\|u(t)\|_{H_Y^1} \le K
\end{equation}
for some constant $K$ and \emph{all} (and not only a.e.) $t\in [0,T]$. In fact, consider a Lebesgue-null set $N$ on whose complement~\eqref{boundeq} holds. Let $t_0\in N$ and $(t_n)_{n\in\mathbb N}\subset (0,T)\setminus N$ converging to $t_0$. Then $\|u(t_n)\|_{H_Y^1} \le K$ for all $n\in\mathbb N$ and by reflexivity of $H^1_Y$ we deduce that (up to taking a subsequence) $(u(t_n))_{n\in\mathbb N}$ converges weakly to some $v$ in $H^1_Y$. It also converges strongly to $u(t_0)$ in $L^2(0,1;H)$, hence by uniqueness of the limit $u(t_0)=v\in H_Y^1$. 
\end{rem}

\section{Invariance properties}

In this section we discuss the issue of invariance
properties of solutions to a network gradient system. The first results
of this kind for solutions of parabolic problems associated with
quadratic forms go back to Beurling and Deny, cf.~
\cite{BeuDen58,BeuDen59}. A similar criterion in the nonlinear case
appear already in Brezis' monograph (cf.~\cite[Thm.~IV.4.5]{Bre73}), but
only recently were Barth\'elemy, Cipriani and Grillo able to prove a more efficient
invariance result exactly in the spirit of Beurling and Deny. Their main
result, which we report for the sake of self-containedness, is the
following (cf.~\cite[Thm.~1.1]{Bar96} and~\cite[Thm.~3.4]{CipGri03}).

\begin{theo}\label{cipgri}
Let $\mathfrak E$ be a lower semicontinuous, convex functional on a
Hilbert space $X$ with values in $(-\infty,\infty]$, and $(T(t))_{t\ge
0}$ be the corresponding strongly continuous, nonexpansive
semigroup on $X$, generated by the subdifferential $\partial \mathfrak E$. Then
$(T(t))_{t\ge 0}$ leaves invariant a closed and convex set $C \subset X$
if and only if 
$${\mathfrak E} (P_C(x)) \leq {\mathfrak E}(x)\quad\hbox{for
all }x \in X,$$ 
where $P_C$ denotes the orthogonal projection of $X$ onto $C$.
\end{theo}
\begin{rem}
Instead of ${\mathfrak E} (P_C(x)) \leq {\mathfrak E}(x)$ for
all $x \in X$, one could of course equivalently require 
$$P_C(x) \in D({\mathfrak E}) \quad \hbox{and} \quad {\mathfrak E} (P_C(x)) \leq {\mathfrak E}(x)\quad\hbox{for
all }x \in D({\mathfrak E}),$$
where $D({\mathfrak E}):=\{x \in X: \mathfrak{E}(x) < \infty \}$ is the effective domain of $\mathfrak{E}$.
\end{rem}
\begin{cor}
In case the effective domain $D({\mathfrak E})=:D$ is a Hilbert space in its own right and the restriction  $\mathfrak{E}_{D}:D \to \mathbb{R}$ is differentiable with derivative ${\mathfrak E}_{D}':D \to D'$, the condition
$$P_C(y) \in D \quad \hbox{and} \quad {\mathfrak E} (P_C(y)) \leq {\mathfrak E}(y)\quad\hbox{for
all }y \in D$$
is equivalent to 
$$P_C(y) \in D \quad \hbox{and} \quad {\mathfrak E}_{D}'(P_C(y))(y-P_C(y)) \geq 0 \quad \hbox{for
all }y \in D.$$ 
\end{cor}
\begin{proof} 
In both directions the proof is an immediate consequence of
$$ 0 \leq {\mathfrak E}'(P_C(y))(y-P_C(y))= \lim_{\lambda \to 0} \frac{{\mathfrak E}(P_C(y)+\lambda(y-P_C(y)))-{\mathfrak E}(P_C(y))}{\lambda} \leq \frac{\lambda{\mathfrak E}(y)-\lambda{\mathfrak E}(P_C(y))}{\lambda}$$
and the fact that $P_C(P_C(y)+\lambda(y-P_C(y)))=P_C(y)$.
\end{proof}
Brezis, Barth\'elemy and Cipriani--Grillo formulate their results in the context of the theory of nonlinear semigroups. In fact, under our standing assumptions it is possible to extend the Lyapunov function $\mathcal E$ defined in~\eqref{lyap} by $+\infty$ to the whole $X_2$. Then the extension of $\mathcal E$ is
lower semicontinuous and convex and its subdifferential is single-valued
and agrees with the gradient $\nabla \mathcal E$ (we will
denote by $\mathcal E$ this extension, too). It is well-known that $\nabla \mathcal E$ generates a semigroup of Lipschitz-continuous mappings $(T(t))_{t\ge 0}$ on $X_2$. If $u_0\in H^1_Y$ and $\psi\equiv 0$, then $u(t)=(T(t))u_0)_{t\ge 0}$ is the solution to
\begin{equation}\tag{GS$_{X_2}$}
\left\{
\begin{array}{rcl}
\dot{u}(t)+\nabla {\mathcal E}(u(t))&=&0,\qquad t\ge 0,\\
 u(0)&=&u_0,   
\end{array}
\right.
\end{equation}
 yielded by Theorem~\ref{wellp}. Standard references for nonlinear semigroup theory include~\cite{Bre73} and~\cite{Miy92}. In particular, invariance of a closed convex set $C$ under the nonlinear semigroup can be rephrased by saying that 
$$\hbox{if }u_0\in C\subset X_2,\quad\hbox{then }u(t)\in C
\hbox{ for all }t\ge 0.$$
Then the following holds.

\begin{theo} \label{order}
Let $H$ be a Hilbert lattice. Under the assumptions of Theorem~
\ref{wellp}, take a closed order interval $\Lambda \subset H$ and consider
the order interval
$${\mathcal C}:=\{f\in X_2: f(x)\in \Lambda \hbox{ for a.e. }x\in
(0,1)\}\subset X_2.$$
Denote by $P_{\mathcal C}$ the orthogonal projection of $X_2$ onto $\mathcal C$ and by $P_{\Lambda}$ the orthogonal projection of $H \times H$ onto $\Lambda \times \Lambda$.
Consider the following assertions.
\begin{enumerate}[(a)]
\item For all initial data $u(0)\in {\mathcal C}\cap H^1_Y$ the
solution $u$ to $\rm(GS_{X_2})$ satisfies $u(t)\in {\mathcal C}$ for all
$t\ge 0$.
\item $P_{\mathcal C}f \in H^1_Y$ and ${\mathcal E}(P_{\mathcal C}f)\le {\mathcal E}(f)$ for all $f\in
H^1_Y$.
\item $P_{\mathcal C}f \in H^1_Y$ and ${\mathcal E}'(P_{\mathcal C}f)(f-P_{\mathcal C}f) \ge 0 $ for all $f\in
H^1_Y$.
\item For all initial data $\xi(0)\in (\Lambda\times \Lambda) \cap Y$ the
solution to
\begin{equation}
\tag{GS$_H$}
\left\{
\begin{array}{rcl}
\dot{\xi}(t)+\phi(\xi(t))&=&0, \qquad t\ge 0,\\
\xi(0)&=&\xi_0,
\end{array}
\right.
\end{equation} 
satisfies $\xi(t)\in \Lambda\times \Lambda$ for all $t\ge 0$.
\item $P_{\Lambda}\xi \in Y$ and $\Phi(P_{\Lambda}\xi)\le \Phi(\xi)$ for all $\xi\in Y$.
\item $P_\Lambda\xi \in Y$ and $\Phi'(P_\Lambda\xi)(P_\Lambda\xi)(\xi - P_\Lambda\xi)\ge 0$ for all $\xi\in Y$.
\end{enumerate}
Then $(a)\Leftrightarrow (b)\Leftrightarrow (c)\Leftarrow (d)\Leftrightarrow (e)\Leftrightarrow (f)$. Moreover, $(c)\Rightarrow (d)$ if ${\mathcal E}_0(P_{\mathcal C}f)= {\mathcal E}_0(f)$ or ${\mathcal E}_0'(P_{\mathcal C}f)(f-P_{\mathcal C}f) = 0 $ for all $f\in H^1_Y$.
\end{theo}

\begin{proof}
The equivalence of $(a),(b)$ and $(c)$ follows from Theorem~\ref{cipgri}. Same for
the equivalence of $(d),(e)$ and $(f)$, once we consider $\rm(GS_{H})$ as a (sub)gradient system on the target space $H \times H$.
Now, observe that
$${\mathcal E}(P_{\mathcal C}f)={\mathcal E}_0(P_{\mathcal C}f)+{\mathcal E}_1(P_{\mathcal C}f),\qquad f\in H^1_Y.$$
We use the fact that, for orthogonal projections $P_{\mathcal C}$ onto
order intervals of $X_2$, $P_{\mathcal C}f \in H^1(0,1;H)$ if $f \in H^1(0,1;H)$ and ${\mathcal E}_0(P_{\mathcal C}u)\le {\mathcal E}_0(u)$ (cf. ~\cite{Mug10}).
Accordingly, a sufficient condition for $(b)$ is given by 
$$ P_\Lambda f_{|\{0,1\}} \in Y \quad \hbox{and} \quad \Phi(P_\Lambda f_{|\{0,1\}})\le \Phi(f_{|\{0,1\}})\qquad\hbox{for all }f\in
H^1_Y.$$
Due to surjectivity of the trace operator $H^1(0,1;H)\to H \times H$, this
condition is equivalent to $(e)$.
\end{proof}

\begin{exa}
We take as order interval on $H$ the positive cone $\Lambda=H^+$ and accordingly on $X_2$ the positive cone ${\mathcal C}=X_2^+$. Assuming  $P_{H^+}\xi= \xi^+ \in Y$ for all $\xi \in Y$, Theorem~\ref{order} asserts that the semigroup $(T(t))_{t \geq 0}$ associated with (GS$_{X_2}$) is \emph{positivity preserving}, i.e. $T(t)(u_0) \in X_2^+$ whenever $u_0 \in X_2^+$, if and only if $\Phi$ satisfies $\Phi(\xi^+) \leq \Phi(\xi)$  for all $\xi \in Y$. The only if part is due to $\mathcal{E}_0'(f^+)(f-f^+)=0$ for all $f \in H^1_Y$; this can be seen by identifying the Hilbert lattice     $H$ with a Lebesgue space $L^2(X)$ for some finite measure space $X$, see e.g.~\cite[Cor.~2.7.5]{Mey91}, via an isometric lattice isomorphism.
\end{exa}

In the nonlinear case, we provide a sufficient condition for well-posedness of a parabolic problems in all $L^p$-spaces by applying a Riesz--Thorin-like interpolation theorem. To this aim, it
is sufficient to show that for any two solutions $u,v$ of $\rm(AV)$
(with initial data $u_0,v_0$) the inequality
$$\|u(t)-v(t)\|_\infty \le \|u_0 - v_0\|_\infty,\qquad t\ge 0,$$
holds. This is equivalent to showing that the closed convex sets
$$\left\{(f,g)\in X_2\times X_2 : \|f-g\|_\infty\le \alpha\right\}$$
are invariant under the solution to the gradient system associated with
the functional
$${\mathscr E}(f,g)={\mathscr E}_0(f,g)+{\mathscr E}_1(f,g)={\mathcal E}(f)+{\mathcal E}(g), \qquad (f,g)\in
H^1_Y\times H^1_Y,$$
where
$${\mathscr E}_0(f,g)={\mathcal E}_0(f)+{\mathcal E}_0(g) \quad \hbox{and} \quad  {\mathscr E}_1(f,g)={\mathcal E}_1(f)+{\mathcal E}_1(g)=\Phi(f_{|\{0,1\}}) + \Phi(g_{|\{0,1\}}). $$
More generally, we are led to considering closed convex subsets \footnote{ Of course, such a set agrees with $C_\alpha$ setting
$\Lambda:=[-\alpha,\alpha]$. This makes sense, once we identify the
Hilbert lattice $H$ with a Lebesgue space $L^2(X)$ for some finite
measure space $X$, see e.g.~\cite[Cor.~2.7.5]{Mey91}, via an isometric
lattice isomorphism.}
$${\mathscr C}_\Lambda:=\left \{(f,g)\in X_2\times X_2 : (f-g)(x)\in \Lambda \hbox{ for a.e. }x\in
(0,1)\right\}$$
of $X_2\times X_2$, where $\Lambda$ is some closed order interval of $H
$. Similarly, we consider the set
$$C_\Lambda:=\{(\eta,\theta)\in (H\times H)\times (H\times H):
(\eta-\theta)\in \Lambda\times \Lambda\}.$$ 
We denote the orthogonal projection of $X_2\times X_2$ (resp.\ $H\times
H$) onto ${\mathscr C}_\Lambda$ (resp. $C_\Lambda$) by $P_{{\mathscr C}_
\Lambda}$ (resp. $P_{C_\Lambda}$).
Then the following holds.

\begin{theo}\label{differences}
Let $H$ be a Hilbert lattice. Under the assumptions of Theorem~
\ref{wellp}, take a closed order interval $\Lambda\subset H$. Consider
the following assertions.
\begin{enumerate}[(a)]
\item For any two initial data $(u(0),v(0))\in \big(H^1_Y \times H^1_Y
\big) \cap {\mathscr C}_\Lambda$ the corresponding solutions $u,v$ to
$\rm(GS_{X_2})$ satisfy $(u(t),v(t))\in  {\mathscr C}_\Lambda$ for all $t
\ge 0$.
\item $P_{{\mathscr C}_\Lambda}(f,g) \in H^1_Y \times H^1_Y$ and  ${\mathscr E}(P_{{\mathscr C}_\Lambda}(f,g))\le {\mathcal
E}(f)+{\mathcal E}(g)$ for all $f,g\in H^1_Y$.
\item $P_{{\mathscr C}_\Lambda}(f,g) \in H^1_Y \times H^1_Y$ and  ${\mathscr E}'(P_{{\mathscr C}_\Lambda}(f,g))((f,g) - P_{{\mathscr C}_\Lambda}(f,g)) \ge 0$ for all $f,g\in H^1_Y$.
\item For any two initial data $(\xi(0),\zeta(0)) \in \big(Y \times Y \big) \cap C_\Lambda $ 
 the corresponding solutions $\xi,\zeta$ to $\rm(GS_H)$ satisfy
$(\xi(t),\zeta(t))\in C_\Lambda $ for all $t\ge 0$.
\item $P_{C_{\Lambda}}(\xi,\zeta) \in Y \times Y$ and $\mathscr{E}_1(P_{C_{\Lambda}}(\xi,\zeta))\le \Phi(\xi) + \Phi(\zeta)$ for all $\xi,\zeta \in Y
$.
\item $P_{C_{\Lambda}}(\xi,\zeta) \in Y \times Y$ and $\mathscr{E}'_1(P_{C_{\Lambda}}(\xi,\zeta))((\xi,\zeta) - P_{C_{\Lambda}}(\xi,\zeta)) \ge 0$ for all $\xi,\zeta \in Y
$.
\end{enumerate}
Then $(a)\Leftrightarrow (b)\Leftrightarrow (c)\Leftarrow (d)\Leftrightarrow (e)\Leftrightarrow (f)$. Moreover, $(c)\Rightarrow (d)$ if 
${\mathscr E}_0(P_{{\mathscr C}_\Lambda}(f,g))= {\mathscr
E}_0(f,g)$ or ${\mathscr E}_0'(P_{{\mathscr C}_\Lambda}(f,g))((f,g) - P_{{\mathscr C}_\Lambda}(f,g)) = 0$ for all $f,g \in H^1_Y$.
\end{theo}

\begin{exa}
The semigroup $(T(t))_{t \geq 0}$ associated with (GS$_{X_2}$) is called \emph{order preserving} if $T(t)(u_{0}) \leq T(t)(v_{0})$ for all $t \geq 0$ and all pairs $(u_{0},v_{0}) \in H_Y^1 \times H_Y^1$ with $u_{0} \leq v_{0}$.\\
Taking as closed, convex set 
$$\mathscr{C}_{H^+}=\{(f,g) \in X_2 \times X_2:  (f-g) \in X_2^+ \}$$ 
and since (cf.~\cite[Lemma 3.3]{CipGri03}) the projection onto $\mathscr{C}_{H^+}$ is given by 
$$ P_{\mathscr{C}_{H^+}}(f,g)=\left(f+\frac{(g-f)^+}{2},g-\frac{(g-f)^+}{2}\right) $$
we can apply Theorem~\ref{differences}. 
This yields for nonnegative $\mathcal{E}$, cf.~\cite[Theorem 3.8]{CipGri03}, the sufficient and necessary condition
\begin{equation}
\label{cipgrieq1}
 \mathcal{E}(f \wedge g) + \mathcal{E}(f \vee g) \leq \mathcal{E}(f) + \mathcal{E}(g)\quad\hbox{for all } f,g \in H_Y^1,
 \end{equation}
for the semigroup $(T(t))_{t \geq 0}$ associated with $\rm(GS_{X_2})$ to be order preserving. In our setting it is easy to see that actually
$$ \mathcal{E}_0(f \wedge g) + \mathcal{E}_0(f \vee g) = \mathcal{E}_0(f) + \mathcal{E}_0(g)$$
whence~\eqref{cipgrieq1} is equivalent to
$$ \Phi(x \wedge y) + \Phi(x \vee y) \leq \Phi(x) + \Phi(y) \quad \hbox{for all } x,y \in Y.$$
\end{exa}

\begin{exa}
In particular, applying the formula for the orthogonal projection onto
$C_\Lambda$ for $\Lambda:=[-\alpha,\alpha]_H$ obtained in~\cite[Lemma~3.3]{CipGri03}
we obtain the following:

Let $H$ be a Hilbert lattice. Assume $\Phi$ and therefore $\mathcal E$ to take values in $[0,\infty)$, hence $\nabla{\mathcal E}$ to be an accretive operator. If
\begin{eqnarray*}
\Phi\left(x+\frac{(x-y+\alpha)^+}{2} -
\frac{(x-y-\alpha)^-}{2}\right)-\Phi(x) \le \Phi(y)- \Phi\left(y-\frac{(x-y+\alpha)^+}{2} +
\frac{(x-y-\alpha)^-}{2}\right)
\end{eqnarray*}
for all $x,y\in Y$ and all $\alpha>0$, then the nonlinear
semigroup associated with $\mathcal E$ is contractive with respect to the norms of both $X_2$ and $L^\infty(0,1;H)$. Then by Browder's
nonlinear interpolation theorem, cf.~\cite[Theorem 3.6]{CipGri03}, one concludes that the nonlinear
semigroup extends to all spaces $L^p(0,1;H)$, $1<p<\infty$.
\end{exa}

\begin{exa}
Another application of the above invariance criterion shows that if
$\Phi$ is additive (resp.\ homogeneous), then so is the solution
operator to $\rm(AV)$: this follows by discussing invariance of the
closed convex sets
$$\left\{(f,g,h):X_2\times X_2\times X_2:f+g=h\right\}$$
(resp.
$$\left\{(f,g):X_2\times X_2:\alpha f=g\right\}\hbox{ for all }\alpha\in
\mathbb R\hbox{)}$$
under the solutions to the product gradient systems associated with
$\mathcal E$.
\end{exa}

\section{Dynamic boundary conditions}

We conclude this note by presenting the natural extension of our setting to the framework of diffusion equations with dynamic, nonlinear boundary conditions. 

Beside~\cite{Bel91,Bel93,AliNic93}, we also mention the more recent article~\cite{FavGolGol06}, where a partial differential inclusion featuring a linear heat equation and a nonlinear dynamic boundary condition is treated. To be more precise, we consider again a closed subspace $Y$ of $H\times H$ and discuss the Cauchy problem
\begin{equation*}\tag{AVD}
\left\{\begin{array}{rcll}
\frac{\partial}{\partial t}{u}(t,x)&=&\frac{\partial^2}{\partial x^2} u(t,x)+\psi_1(t,x),\qquad &t\in [0,T], \quad x\in (0,1),\\
u(t)_{|\{0,1\}}&\in& Y, &t\in [0,T],\\
\frac{\partial }{\partial t}u(t)_{|\{0,1\}}&=&-P_Y\left(\frac{\partial u(t)}{\partial\nu}+\phi\left(( u(t)_{|\{0,1\}}\right)\right)+\Psi_2(t), &t\in [0,T], \quad \Psi_2:[0,T] \to Y,\\
u(0)&=&u_0,\\
u(0)_{|\{0,1\}}&=&v_0.
\end{array}
\right.
\end{equation*}
Observe that if $Y={\rm Range}\,\tilde{ I}$ for a certain incidence matrix $I$ (see Example~\ref{quantumgr}), $\rm(AVD)$ is a nonlinear generalisation of the setting considered in~\cite{MugRom07}.

We consider basically the same functional $\mathcal E$ introduced in Section~\ref{sec2} but on a different Hilbert space. In fact, we introduce the space
$${\mathcal H}^1_Y:=\left\{{\mathfrak f}:=\begin{pmatrix}
f\\ g
\end{pmatrix}\in H^1_Y \times H\times H: f_{|\{0,1\}}=g\right\}.$$
and consider the functional $\mathcal  E :{\mathcal H}^1_Y \to \mathbb R$. The reference space is now the Hilbert product space ${\mathcal X}_2:=L^2(0,1;H)\times Y$.

In this framework the diffusion equation $\rm{(AVD)}$ corresponds to the gradient system
\begin{equation}\tag{GS$_{{\mathcal X}_2}$}
\left\{
\begin{array}{rcl}
\dot{{\mathfrak u}}(t)+\nabla {\mathcal E}({\mathfrak u}(t))&=&\Psi(t),\qquad t\in[0,T],\\
 {\mathfrak u}(0)&=&{\mathfrak u}_0,   
\end{array}
\right.
\end{equation}
with respect to the Hilbert spaces ${\mathcal H}^1_Y$ and ${\mathcal X}_2:=X_2\times Y$, where 
$${\mathfrak u}_0:=\begin{pmatrix} u_0 \\ v_0 \end{pmatrix} \quad \hbox{and} \quad \psi(t):=\begin{pmatrix}
\Psi_1(t)\\ \Psi_2(t)
\end{pmatrix},\quad \Psi_1(t):= \psi_1(t,\cdot),\quad t\in [0,T].$$

 To see this, we assume first $\mathfrak f :=\begin{pmatrix}
f\\ f_{|\{0,1\}} \end{pmatrix} \in  D(\nabla {\mathcal E})$ and choose $\mathfrak g:=\begin{pmatrix}
g\\ g_{|\{0,1\}} \end{pmatrix} \in \mathcal{H}^1_Y$ with  $g_{|\{0,1\}}=0 $. Thus, we obtain from
$$\mathcal E^\prime(\mathfrak f) \mathfrak g= \int_0^1 \langle f^\prime(x)\,,\,g^\prime(x)\rangle_H \, dx + \langle \phi(f_{|\{0,1\}})\,,\,0 \rangle_{H \times H} \stackrel{!}{=}\int_0^1{\langle \xi_0\,,\, g \rangle_{H} + \langle \xi_1\,,\, 0} \rangle_{H \times H}$$
that $f \in H^2(0,1;H)$ and $f^{\prime \prime}=-\xi_0$ and hence for arbitrary $\mathfrak g:=\begin{pmatrix} g\\ g_{|\{0,1\}} \end{pmatrix} \in \mathcal{H}^1_Y$, after having integrated by parts,
$$  \langle  \frac{\partial f}{\partial\nu}\,,\,g_{|\{0,1\}} \rangle_{H \times H} +\langle \phi(f_{|\{0,1\}})\,,\,g_{|\{0,1\}}\rangle_{H\times H} =  \langle \xi_1\,,\, g_{|\{0,1\}} \rangle_{H \times H}.$$
We infer $\xi_1=P_Y\left(\frac{\partial f}{\partial\nu}+\phi\left(( f_{|\{0,1\}}\right)\right)$.\\
For $f \in H^2(0,1;H) \cap H_Y^1$ one can define $ \xi_0$ and $\xi_1$ as above in order to see that $\mathfrak f :=\begin{pmatrix}
f\\ f_{|\{0,1\}} \end{pmatrix} \in D(\nabla {\mathcal E})$.  

\begin{theo}\label{wellp2}
Let $\phi:Y\to Y$ be a function that satisfies $\phi=\nabla \Phi$ for some convex $\Phi\in C^1(Y,{\mathbb R})$ and that maps bounded sets into bounded sets. 
Then for all initial data ${\mathfrak u} \in {\mathcal H}^1_Y$ there exists a unique solution $u\in H^1(0,T;X_2) \cap L^\infty(0,T;H^1_Y)$ to $\rm{(AVD)}$. If $\Psi \equiv 0 $, then the solution also belongs to $L^2(0,T;H^2(0,1;H))$ for every $T \geq 0$ and even to $L^2(0,\infty;H^2(0,1;H))$ if in addition $\Phi$ and therewith $\mathcal{E}$ is bounded from below. 
\end{theo}

\begin{proof}
The proof is again based on the application of Theorem~\ref{gengrad} to the Hilbert spaces ${\mathcal H}^1_Y$ and ${\mathcal X}_2$ and to the functional $\mathcal  E :{\mathcal H}^1_Y \to \mathbb R$. 
One sees that $\mathcal E$  satisfies~\eqref{growthcond}. It follows from our assumptions that $\mathcal E'$ maps bounded sets of ${\mathcal H}^1_Y$ into bounded sets of $({\mathcal H}^1_Y)'$. Accordingly,  for all initial data ${\mathfrak u} \in {\mathcal H}^1_Y$ there exists a unique solution ${\mathfrak u}\in H^1(0,T;{\mathcal X}_2)\cap L^\infty(0,T;{\mathcal H}^1_Y)$ to $\rm{(AVD)}$. Taking the first coordinate of $\mathfrak u$ yields the claim.
\end{proof}

\bibliographystyle{alpha}
\bibliography{/home/delio/Dropbox/mate/referenzen/literatur}
\end{document}